\documentclass{amsart}
\newcommand*{\thisisamsart}{}
\usepackage{amsfonts}
\usepackage{amsmath}
\usepackage{amsthm}
\usepackage{amssymb}
\usepackage{amscd}
\usepackage{amsxtra}
\usepackage{enumitem}
\usepackage{thmtools}
\usepackage{thm-restate}
\usepackage{microtype}
\theoremstyle{plain}
\usepackage{lmodern}
\usepackage{xcolor}
\usepackage[hyphens]{url}
\usepackage{hyperref}

\hypersetup{
    colorlinks,
    citecolor={red!50!black},
    linkcolor={blue!50!black},
    urlcolor={blue!80!black}
}

\newtheorem{thm}{Theorem}[subsection]
\numberwithin{thm}{section}
\newtheorem{lem}[thm]{Lemma}
\newtheorem{prop}[thm]{Proposition}

\newtheorem{cor}[thm]{Corollary}
\theoremstyle{definition}
\newtheorem{defn}[thm]{Definition}

\theoremstyle{remark}

\newtheorem{eg}[thm]{Example}

\newcommand{\angles}[1]{\left\langle #1 \right\rangle}

\ifdefined\thisisamsart
\newcommand{\ead}{\email}
\fi 
\begin{document}

\title{Extension of valuations in characteristic one}
\author{Jeffrey Tolliver}
\ead{tolliver@ihes.fr}
\address{Institut des Hautes \'Etudes Scientifiques, 35 Route de Chartres\\ Bures-sur-Yvette France 91440}

\begin{abstract}We develop an extension of valuations theorem for suitable extensions of idempotent semirings.  As an application, we give a new proof for the classical case of fields.  Along the way, we develop characteristic one analogues of some central results in the theory of valuation rings.\end{abstract}
\maketitle

\section*{Acknowledgements}

I'd like to thank Alain Connes for pointing out some typos in a previous version of this manuscript.

\section{Introduction}

The goal of this paper is to prove an analogue of the theorem on extension of valuations for idempotent semifields, which are defined as follows.

\begin{defn}An abelian semigroup is called \emph{idempotent} if it satisfies $x+x=x$ for all $x$.  A semiring\footnote{In this paper, the word semiring will refer to a commutative semiring.} or semifield is said to be \emph{idempotent} (or said to have \emph{characteristic one}) if the underlying additive semigroup is idempotent.
\end{defn}

Idempotent semirings have occured in connection with non-archimedean geometry, tropical geometry and the study of the Riemann zeta function (c.f. \cite{giansiracusa},\cite{macpherson},\cite{arithmeticsite},\cite{scalingsite}).  A simple example is the semiring of ideals of a ring.  Another important example is the semifield $\Gamma_{\max}=\Gamma\cup\{0\}$ associated to a totally ordered abelian group $\Gamma$.  In this semifield, the multiplication comes from the group structure, while addition is given by declaring $x+y=\max(x,y)$.  A key property of idempotent semigroups is that they come with a canonical partial order: $x\leq y$ if $x+y=y$.  Furthermore, the sum of a collection of elements is simply their least upper bound.  In the case of a totally ordered idempotent semifield, this gives a total order on the multiplicative group, so that every totally ordered idempotent semifield is of the form $\Gamma_{\max}$ up to isomorphism.  

For the present purposes, totally ordered idempotent semifields are crucial because they are naturally the codomains of valuations --- a valuation with value group $\Gamma$ can be seen as actually landing in $\Gamma_{\max}$.  J. and N. Giansiracusa have shown in \cite{giansiracusa} that there is a semiring $\mathcal{S}_f(R)$ associated to any ring $R$ with the property that valuations on $R$ are the same as homomorphisms from $\mathcal{S}_f(R)$ to totally ordered idempotent semifields.  In particular, multiplicative seminorms (i.e. valuations with value group $\mathbb{R}$) are the same as homomorphisms $\mathcal{S}_f(R)\rightarrow\mathbb{R}_{\max}$ where $\mathbb{R}_{\max}$ is the so-called \emph{tropical semifield}.  This semifield occurs in connection with tropical geometry in \cite{giansiracusa}, where it is shown that points on a tropical variety correspond to homomorphisms from its coordinate semiring to $\mathbb{R}_{\max}$.

The above examples hint that homomorphisms into totally ordered idempotent semifields are important when studying valuations on rings.  This might suggest that they are related to valuations on idempotent semirings.  The following definition is a natural generalization of the classical one.

\begin{defn}Let $R$ be a semiring.  A valuation on $R$ is a map $v:R\rightarrow\Gamma_{\max}$ for some totally ordered abelian group $\Gamma$ such that the following properties hold.

\begin{enumerate}
 \item $v(0)=0$ and $v(1)=1$.
 \item $v(xy)=v(x)v(y)$ for all $x,y\in R$.
 \item\label{90} $v(x+y)\leq v(x)+v(y)$ for all $x,y\in R$.
 \item \label{91} $v(x)\leq v(x+y)+v(y)$ for all $x,y\in R$.
\end{enumerate}

The value group of the valuation is the subgroup of $\Gamma$ generated by the nonzero elements of the image of $v$.
\end{defn}

Axiom \ref{90} is the ultrametric inequality.  In the case of rings, Axiom \ref{91} implies $v(x)=v(-x)$.  In the case of idempotent semirings, Axioms \ref{90} and \ref{91} together are equivalent to $v(x+y)=v(x)+v(y)$, which implies that valuations on an idempotent semiring $R$ with value group $\Gamma$ are precisely the surjective homomorphisms $R\rightarrow\Gamma_{\max}$.  One interpretation of Axiom \ref{90} is that $\{x\in R\mid v(x)\leq a\}$ is a subsemigroup.  From this point of view, Axiom \ref{91} says that $\{x\in R\mid v(x)\leq a\}$ is in fact a saturated subsemigroup (c.f. Definition \ref{105b}), which is a useful property to have if one desires to use the valuation to construct quotient objects (as is done in local number theory).

The most fundamental result in the theory of valuations states that given a valuation $v:K\rightarrow \Gamma_{\max}$ and an extension $L$ of $K$, we may extend $v$ to $L$.  More precisely, there exists a group $\Gamma'\supseteq \Gamma$ and a valuation $w:L\rightarrow \Gamma'_{\max}$ such that $w\mid_K=v$.  We shall show in Corollary \ref{608} that the same result is true for idempotent semifields, and shall give partial results in the more general setting of unitgenerated idempotent semirings (c.f. Theorem \ref{518}).  As an application, we will obtain a new proof of the extension of valuations theorem for fields by reduction to the case of unitgenerated idempotent semirings.  Along the way in Sections \ref{s143} and \ref{s392}, we also develop other results of valuation theory in characteristic one, for instance a theory of valuation subsemirings of a valued idempotent semifield.

\section{Definitions}
\newcommand{\sfinite}{finite}

\begin{defn}\label{105b}Let $M$ be an semigroup.  Then a subsemigroup $N\subseteq M$ is \emph{saturated} if $x\in N$ and $x+y\in N$ imply that $y\in N$.
\end{defn}

Saturated subsemigroups are precisely those which arise as kernels of homomorphisms.

\begin{prop}Let $M$ be an idempotent semigroup.  Then a subsemigroup $N\subseteq M$ is saturated if $x\leq y$ and $y\in N$ imply that $x\in N$.
\end{prop}

\begin{defn}Let $R$ be an idempotent semiring and $M$ be an $R$-module.  Then $M$ is \emph{\sfinite} if there is some $x\in M$ such that for all $y\in M$ there exists $r\in R$ with $y\leq rx$.
\end{defn}

A module is {\sfinite} if it is generated as a saturated submodule of itself by a single element.  It turns out that the saturated submodule generated by a finite collection of elements is generated by their sum.  Hence a module is {\sfinite} if it is the smallest saturated submodule of itself containing some specific finite subset.

\begin{defn}Let $R$ be an idempotent semiring and $M$ be an $R$-module.  $M$ is noetherian if every submodule is {\sfinite}.   $M$ is seminoetherian if every {\sfinite} submodule is noetherian.
\end{defn}

\begin{defn}Let $R$ be an idempotent semiring.  $R$ is \emph{simple} if the only saturated ideals are $0$ and $R$.\end{defn}

An easy argument shows the following.

\begin{prop}\label{87}Let $R$ be an idempotent semiring.  $R$ is simple if and only if for every $x\in R$ there is some $y\in R$ with $xy\geq 1$.
\end{prop}

\begin{defn}An idempotent semiring $K$ is \emph{unitgenerated} if every element is a finite sum of units.
\end{defn}

Any idempotent semifield is unitgenerated.  The following result shows any unitgenerated idempotent semiring is simple.

\begin{prop}Let $R$ be a unitgenerated idempotent semiring.  Then $R$ is simple.
\end{prop}

\begin{proof}Let $x\in R$.  Then $x=\sum_{i\in I}u_i$ where $I$ is a finite set and each $u_i$ is a unit.  Hence $x\geq u_i$ so $xu_i^{-1}\geq 1$.  Proposition \ref{87} implies the result.
\end{proof}

\begin{eg}Let $K$ be an idempotent semifield.  Then the semiring of Laurent polynomials $K[x_1,x_1^{-1},\ldots,x_n,x_n^{1}]$ is unitgenerated.  Furthermore, any quotient of this semiring is unitgenerated as well.  If $K=\mathbb{R}_{\max}$ then quotients of the semiring of Laurent polynomials occur naturally in the framework of tropical geometry.
\end{eg}

Since simple semirings have no saturated ideals, they seem a natural generalization of fields.  Semifields are of course also a natural generalization of fields.  Because idempotent semifields are unitgenerated and unitgenerated idempotent semirings are simple, this suggests that unitgenerated semirings should also be viewed as an analogue of fields.  The following example yields a connection between unitgenerated idempotent semirings and fields.

\begin{eg}Let $R$ be a ring and $A$ be a ring containing $R$.  Let $\mathcal{S}_f(A,R)$ be the semiring of finitely generated $R$-submodules of $A$ (c.f. Definition \ref{627} and the comments following it).  If $A$ is a field, then $\mathcal{S}_f(A,R)$ is unitgenerated since the nonzero principal submodules are all units.  If $\mathcal{S}_f(A,R)$ is unitgenerated (or even simple), then applying Proposition \ref{87} to $xR$ (for any nonzero $x\in A$ shows there is a submodule $M$ such that $1\in R\subseteq xM$.  This implies that $x$ is a unit.  Hence $\mathcal{S}_f(A,R)$ is unitgenerated if and only if $A$ is a field, which in turn occurs if and only if $A$ is simple.
\end{eg}

\section{Valuation semirings}\label{s143}

\begin{lem}\label{164}Let $K$ be a unitgenerated idempotent semiring and $\Gamma$ a totally ordered abelian group.  Let $v:K\rightarrow \Gamma_{\max}$ be a homomorphism.  Define a relation $\preceq$ on $K^\times$ by $x\preceq y$ if $v(x)\leq v(y)$.  Then $\preceq$ makes $K^\times$ into a preordered abelian group.
\end{lem}

\begin{prop}Let $K$ be a unitgenerated idempotent semiring and $\Gamma$ a totally ordered abelian group.  Let $v:K\rightarrow \Gamma_{\max}$ be a homomorphism.  Let $R=\{x\in K\mid v(x)\leq 1\}$.  Let $\preceq$ be as in Lemma \ref{164}.  $\preceq$ induces a well-defined total order on $K^\times/R^\times$.  Furthermore, $v$ induces an isomorphism $K^\times/R^\times\cong v(K^\times)\subseteq\Gamma$ of totally ordered abelian groups.
\end{prop}

\begin{proof}Clearly $R^\times=\{x\in K^\times\mid v(x)=1\}$.  By replacing $\Gamma$ with $v(K^\times)$, we may assume without loss of generality that $v(K^\times)=\Gamma$.  Restricting $v$ gives a surjective group homomorphism $K^\times\rightarrow \Gamma$ which induces an isomorphism $K^\times/R^\times\cong \Gamma=v(K^\times)$.  Let $[x],[y]\in K^\times/R^\times$ be the classes of $x,y\in K$.  Then $[x]\preceq [y]$ if and only if $v(x)\leq v(y)$.  Hence the desired isomorphism and its inverse preserve the order.
\end{proof}

\begin{defn}A valuation subsemiring $R$ of a unitgenerated idempotent semiring $K$ is a semiring $R$ such that there exists an embedding $R\rightarrow K$ of $R$ as a saturated subsemiring of $K$ such that for every $x\in K^\times$, either $x\in R$ or $x^{-1}\in R$.
\end{defn}

\begin{prop}Let $K$ be a unitgenerated idempotent semiring and $\Gamma$ a totally ordered abelian group.  Let $v:K\rightarrow \Gamma_{\max}$ be a homomorphism.  Let $R=\{x\in K\mid v(x)\leq 1\}$.  Then $R$ is a valuation subsemiring of $K$.
\end{prop}

We would like to prove the converse of the above proposition.  When $R$ is a valuation subsemiring of a unitgenerated idempotent semiring $K$, we define a relation $\preceq$ on $K^\times$ by $x\preceq y$ if $xy^{-1}\in R$.

\begin{lem}Let $R$ be a valuation subsemiring of a unitgenerated idempotent semiring $K$.  Then $\preceq$ makes $K^\times$ into a preordered abelian group.
\end{lem}

\begin{lem}Let $R$ be a valuation subsemiring of a unitgenerated idempotent semiring $K$.  Then $\preceq$ induces a well-defined relation on $K^\times/R^\times$ making $K^\times/R^\times$ into a totally ordered abelian group.
\end{lem}

\begin{proof}For $x,y\in K^\times$ it follows from the definition of a valuation subsemiring that either $xy^{-1}\in R$ or $x^{-1}y\in R$.  Hence, either $x\preceq y$ or $y\preceq x$.  Observe furthermore that $x\preceq y\preceq x$ if and only if $xy^{-1}\in R^\times$, which occurs precisely when $x,y$ define the same element in the quotient $K^\times/R^\times$.  From these facts and the fact that $\preceq$ makes $K^\times$ into a preordered abelian group, we may obtain the result.
\end{proof}

We will sometimes write $\leq$ instead of $\preceq$ for the induced order on $K^\times/R^\times$, since there is no danger of confusing it with the order on $K$.  We are now ready to construct the valuation on $K$ induced by its valuation subsemiring.

\begin{prop}\label{193}Let $R$ be a valuation subsemiring of a unitgenerated idempotent semiring $K$.  Let $\Gamma=K^\times/R^\times$, and for any $u\in K^\times$, we write $[u]\in K^\times/R^\times$ for its class.  Let $v:K\rightarrow \Gamma_{\max}$ be given by $v(\sum_{i\in I} x_i)=\sum_{i\in I}[x_i]$ for any finite set $I$ and any collection of units $x_i\in K^\times$.  Let $x,y\in K$ with $y\in K^\times$.  Then $v(x)\leq [y]$ if and only if $xy^{-1}\in R$.  In particular, $v$ is well-defined.  Furthermore, $v$ is a surjective homomorphism of semirings.
\end{prop}

\begin{proof}
Fix a decomposition $x=\sum_{i\in I}x_i$ as a sum of elements $x_i\in K^\times$ indexed by a finite set $I$.  Suppose $xy^{-1}\in R$.  Then since $x_iy^{-1}\leq xy^{-1}$, it follows that $x_iy^{-1}\in R$ for each $i$.  We may rewrite this as $[x_i]\preceq [y]$.  Then $\sum_{i\in I}[x_i]\leq [y]$ inside $\Gamma_{\max}$, so $v(x)\leq [y]$.  Now suppose that $v(x)\leq [y]$.  Then $\sum_{i\in I}[x_i]\leq [y]$.  Hence for each $i$, $[x_i]\leq [y]$, or equivalently $x_iy^{-1}\in R$.  Hence $xy^{-1}=\sum_{i\in I}x_iy^{-1}\in R$.  

To see that $v$ is well-defined, define $v'(x)$ in the same way as $v(x)$, but using a different sum decomposition.  If $x=0$, the only sum decomposition is the empty one, and there is nothing to prove.  Otherwise, we are decomposing $x$ as a sum over a nonempty set, so $v(x),v'(x)\neq 0$.  Then $v(x),v'(x)\in \Gamma=K^\times/R^\times$.  Hence there is some $y$ such that $v'(x)=[y]$.  Since $v'(x)\leq [y]$, $xy^{-1}\in R$, so $v(x)\leq [y]=v'(x)$.  The reverse inequality is similar.  

It is clear from the definition that $v$ is a semigroup homomorphism.  To check it preserves multiplication, note that by the distributive law, we may check this on a generating set, and in particular on $K^\times$.  Let $x,y\in K^\times$.  Then $v(xy)=[xy]=[x][y]=v(x)v(y)$.  It is also easy to see $v(1)=1$.  To see that it is surjective, note that it is surjective even when restricted to $K^\times\cup\{0\}$.
\end{proof}

\begin{thm}Let $R$ be a valuation subsemiring of a unitgenerated idempotent semiring $K$.  Let $\Gamma$ and $v$ be as in Proposition \ref{193}.  Then $R=\{x\in K\mid v(x)\leq 1\}$.
\end{thm}

\begin{proof}Adopt the notation of Proposition \ref{193}.  Let $x\in K$, and write $x=\sum_{i\in I}x_i$ as a sum of elements $x_i\in K^\times$ indexed by a finite set $I$.  Suppose $x\in R$.  Then $x_i\in R$ for all $i\in I$, so $[x_i]\leq 1$.  Then $v(x)=\sum_{i\in I}[x_i]\leq 1$.  Conversely, suppose $v(x)\leq 1$.  Then for each $i$, $[x_i]\leq 1$, so $x_i\in R$.  Then $x=\sum_{i\in I}x_i\in R$.
\end{proof}

Our next goal is to understand the ideals of $R$ and the fractional ideals of $K$.

\begin{thm}\label{150}Let $R$ be a valuation subsemiring of a unitgenerated idempotent semiring $K$. Let $\Gamma$ and $v$ be as in Proposition \ref{193}.  To a saturated subsemigroup $U\subseteq \Gamma_{\max}$, associate the saturated $R$-submodule $v^{-1}(U)$ of $K$.  This yields a bijective order-preserving correspondence between saturated subsemigroups of $\Gamma_{\max}$ and saturated $R$-submodules of $K$.  The inverse correspondence sends $M\subseteq K$ to $v(M)$.
\end{thm}

\begin{proof}For the moment, let $U\subseteq \Gamma_{\max}$ be a saturated subsemigroup.  To see that $v^{-1}(U)$ is a saturated $R$-submodule, note first that it is clearly a saturated subsemigroup.  Let $x\in v^{-1}(U)$ and $r\in R$.  Then $v(r)\leq 1$, so $v(rx)\leq v(x)$.  Since $v(x)\in U$, we conclude that $v(rx)\in U$ so $rx\in v^{-1}(U)$, and $U$ is a saturated $R$-submodule.

Let $M\subseteq K$ be a saturated $R$-submodule.  Let $x\in M$ and $y\in K$ with $v(y)\leq v(x)$.  I claim that $y\in M$.  Write $x=\sum_{i\in I}x_i$ with $x_i\in K^\times$ and $I$ finite.  For each $i$, $x_i\in M$.  Choose $k\in I$ to maximize $v(x_k)$ and observe that $v(y)\leq v(x)=\max_{i\in I} v(x_i)=v(x_k)$.  Hence in proving this claim we may assume with out loss of generality that $x=x_k\in K^\times\cap M$.  Now $v(x)=[x]$, and by Proposition \ref{193}, $yx^{-1}\in R$.  Since $M$ is an $R$-submodule, $y=(yx^{-1})x\in M$, as claimed.  

Now let $U$ be the image of $M$ under $v$.  $U\subseteq\Gamma_{\max}$ is a subsemigroup because $\Gamma_{\max}$ is totally ordered.  Suppose $u\in U$ and $t\leq u$.  We may choose $x\in M$ so $v(x)=u$.  Since $v$ is surjective, we may also choose $y\in K$ so $t=v(y)\leq v(x)$.  By the claim above, $y\in M$ so $t\in U$.  Hence $U$ is saturated.  Clearly $M\subseteq v^{-1}(v(M))=v^{-1}(U)$.  Let $y\in v^{-1}(U)$.  Then $v(y)\in U$ so there exists $x\in M$ with $v(y)=v(x)$.  By the claim above, $y\in M$.  Hence $M=v^{-1}(U)$, so the correspondence in the statement of the theorem is surjective.  Since $v$ is surjective, $v(v^{-1}(U))=U$.  Hence if $v^{-1}(U)=v^{-1}(V)$ then $U=V$, so the construction is injective.  The correspondence is order-preserving because it is given by taking the preimage.
\end{proof}

In particular, we may describe the saturated ideals of $R$.

\begin{cor}\label{163}Let $R$ be a valuation subsemiring of a unitgenerated idempotent semiring $K$. Let $\Gamma$ and $v$ be as in Proposition \ref{193}.  There is a bijective correspondence between saturated ideals of $R$ and saturated subsemigroups of $\{x\in \Gamma_{\max}\mid x\leq 1\}$.
\end{cor}

\begin{proof}Saturated ideals of $R$ are the same as saturated $R$-submodules $X\subseteq K$ such that $X\subseteq R$.  Under the correspondence of Theorem \ref{150}, these correspond to saturated subemigroups $Y\subseteq \Gamma_{\max}$ such that $Y\subseteq v(R)=\{x\in \Gamma_{\max}\mid x\leq 1\}$.
\end{proof}

\begin{cor}Let $R$ be a valuation subsemiring of a unitgenerated idempotent semiring $K$.  The ideals of $R$ are totally ordered.
\end{cor}

\begin{proof}Let $\Gamma$ be as in Proposition \ref{193}.  By Corollary \ref{163}, it suffices to show that the saturated subsemigroups of $\{x\in \Gamma_{\max}\mid x\leq 1\}$ are totally ordered.  Let $E=\{x\in \Gamma_{\max}\mid x\leq 1\}$.  Let $I,J\subseteq E$ be saturated subsemigroups.  We wish to show that either $I\subseteq J$ or $J\subseteq I$.  So suppose $J\not\subseteq I$.  Choose $x\in I$ with $x\not\in J$.  Let $y\in J$.  If $x\leq y$ then $x$ would be in $J$, so we have $x\not\leq y$.  Since $E$ is totally ordered, $y\leq x$.  Hence $y\in I$.  Since this holds for all $y\in J$, we have $I\subseteq J$.
\end{proof}

\section{Radical congruences on simple idempotent semirings}

Let $K$ be a simple idempotent semiring and $x,y\in K$.  Our goal is to determine when $v(x)=v(y)$ for every valuation $v$.  This will be used later, when we write the integral closure as an intersection of valuation semirings.  The below definitions come from \cite{joomin}.

\begin{defn}Let $R$ be an idempotent semiring.  A congruence $\sim$ on $R$ is \emph{prime} if $xy+zw\sim xw+zy$ implies that either $x\sim z$ or $y\sim w$ and if $0\not\sim 1$.  $\sim$ is \emph{QC} (or quotient cancellative) if $R/\mathnormal{\sim}$ is cancellative.  $\sim$ is \emph{radical} if it is the intersection of prime congruences.  $R$ is a \emph{domain} if equality is a prime congruence.
\end{defn}

\begin{lem}Let $f:R\rightarrow A$ be a morphism of idempotent semirings.  Let $\sim$ be a congruence on $A$.  Then the congruence $\equiv=f^{-1}(\sim)$ on $R$ is prime if $\sim$ is prime.  When $f$ is surjective and $\equiv$ is prime, then $\sim$ is prime.  In particular, if $\equiv$ is any congruence on $R$, then $R/\equiv$ is a domain if and only $\equiv$ is prime.
\end{lem}

\begin{proof}The last statement comes from taking $\sim$ to be equality in the first two statements, so we need only prove them.  Suppose $\sim$ is prime.  Suppose $x,y,z,w\in R$ satisfy $xy+zw\equiv xw+zy$.  Then $f(x)f(y)+f(z)f(w)\sim f(x)f(w)+f(z)f(y)$.  Since $\sim$ is prime, either $f(x)\sim f(z)$ or $f(y)\sim f(w)$ and hence either $x\equiv z$ or $y\equiv w$.  Conversely, suppose $\equiv$ is prime and $f$ is surjective.  Let $x,y,z,w\in A$ be such that $xy+zw\sim xw+zy$.  Pick $\hat{x},\hat{y},\hat{z},\hat{w}\in R$ to be lifts (e.g. $f(\hat{x})=x$).  Then $\hat{x}\hat{y}+\hat{z}\hat{w}\equiv\hat{x}\hat{w}+\hat{z}\hat{y}$ so $\hat{x}\equiv\hat{z}$ or $\hat{y}\equiv\hat{w}$.  This implies $x\sim z$ or $y\sim w$ so $\sim$ is prime.
\end{proof}

The following two theorems come from \cite{joomin}.

\begin{thm}\label{248}QC congruences are radical.
\end{thm}

\begin{thm}\label{193b}An idempotent semiring is a domain if and only if it is totally ordered and cancellative.
\end{thm}

\begin{cor}Let $\sim$ be a congruence on an idempotent semiring $R$.  $\sim$ is prime if and only if $R/\mathnormal{\sim}$ is totally ordered and cancellative.  In particular, prime congruences are QC.
\end{cor}



We will also need a converse to Theorem \ref{248}.  There is no reason to expect this is true in general, but it is when $R$ is simple.

\begin{lem}Let $R$ be a simple idempotent semiring.  Radical congruences on $R$ are QC.
\end{lem}

\begin{proof}Suppose $\sim$ is radical and that $sx=sy$ for $s\neq 0$.  Then for each prime congruence $\equiv$ containing $\sim$, we have $sx\equiv sy$.  Since $\equiv$ is QC we conclude either $s\equiv 0$ or $x\equiv y$.  If $s\equiv 0$, then $s$ is in the kernel of $R\rightarrow R/\equiv$, which is a proper saturated ideal.  Since $R$ is simple, this implies $s=0$, a contradiction.  So we have $x\equiv y$ for each prime congruence containing $\sim$.  But $\sim$ is the intersection of such congruences so $x\sim y$.  Hence $\sim$ is QC.
\end{proof}

\begin{defn}Let $R$ be an idempotent semiring.  Its \emph{reduction}, \emph{$R_{\mathrm{red}}$} is the quotient of $R$ by the intersection of all prime congruences.  $R$ is \emph{reduced} if the intersection of all prime congruences is equality.
\end{defn}

\begin{cor}\label{215}Let $R$ be a simple idempotent semiring.  A quotient $R/\mathnormal{\sim}$ of $R$ is reduced if and only if $R/\mathnormal{\sim}$  is cancellative.  
\end{cor}

We have the following result.

\begin{prop}\label{225}Let $R$ be an idempotent semiring.  Then $R_\mathrm{red}$ is reduced, and every surjective homomorphism from $R$ to a reduced idempotent semiring factors uniquely through $R_{\mathrm{red}}$.
\end{prop}

\begin{proof}
I claim that an idempotent semiring $A$ is reduced if and only if homomorphisms from $A$ into domains separate points.  More precisely, $A$ is reduced if and only if for all $x,y\in A$ with $x\neq y$, there is a domain $B$ and homomorphism $f:A\rightarrow B$ such that $f(x)\neq f(y)$.  Suppose first that $A$ is reduced.  Since equality is the intersection of prime congruences and since $x\neq y$, there is some prime congruence $\sim$ with $x\not\sim y$.  One can take $B=A/{\sim}$ and take $f$ to be the quotient map.  Conversely, suppose homomorphisms into domains separate points.  Pick $x,y\in A$ so $x\sim y$ for every prime congruence $\sim$.  For the sake of contradiction, suppose $x\neq y$.  Then we may pick a homomorphism $f:A\rightarrow B$ into a domain $B$ such that $f(x)\neq f(y)$.  By Theorem \ref{193b}, $f(A)\subseteq B$ is also a domain.  Define $\sim$ by declaring $a\sim b$ to mean $f(a)=f(b)$. Note that $A/{\sim}\cong f(A)$ is a domain, so $\sim$ is prime.  Observe that $x\not\sim y$, which is a contradiction.  Hence $x=y$ and $A$ is reduced, which establishes the claim.

Let $\sim$ be the intersection of all prime congruences in $R$, so $R_\mathrm{red}=R/{\sim}$.  Let $x,y\in R/{\sim}$ be distinct.  Let $\hat{x},\hat{y}\in R$ be lifts.  Then $\hat{x}\not\sim\hat{y}$.  Then there is some prime congruence $\equiv$ in $R$ such that $\hat{x}\not\equiv\hat{y}$.  Let $\pi:R\rightarrow R/\equiv$ and $\phi:R\rightarrow R/{\sim}$ be the quotient maps.  Since $\sim$ is in $\equiv$, we have a quotient map $\psi:R/{\sim}\rightarrow R/\equiv$ with $\pi=\psi\phi$.  We know $\psi(x)=\psi(\phi(\hat{x}))=\pi(\hat{x})\neq \pi(\hat{y})=\psi(y)$ so $\psi$ separates $x$ and $y$.  Since $\equiv$ is prime, $R/\equiv$ is a domain.  Hence, homomorphisms from $R_\mathrm{red}=R/{\sim}$ into domains separate points.

Let $A$ be a reduced idempotent semiring and $f:R\rightarrow A$ be a homomorphism.  Let $\sim$ be the intersection of all prime congruences.  Let $x,y\in A$.  We wish to show that if $x\sim y$ then $f(x)=f(y)$.  We will prove the contrapositive, so we assume $f(x)\neq f(y)$.  Then there is a morphism $g:A\rightarrow B$ into a domain $B$ such that $g(f(x))\neq g(f(y))$.  Let $C$ be the image of $g\circ f$ and let $h:R\rightarrow C$ be induced by $g\circ f$.  Then $C\subseteq B$ is a domain since $B$ is a domain.  Furthermore $h$ is surjective and $h(x)\neq h(y)$.  Define $\equiv$ by declaring $u\equiv v$ to mean $h(u)=h(v)$.  Then $\equiv$ is a prime congruence and $x\not\equiv y$.  Hence $x\not\sim y$.
\end{proof}

\begin{cor}Let $R$ be a simple idempotent semiring.  Then $R_\mathrm{red}$ is cancellative and every morphism from $R$ to a cancellative idempotent semiring factors uniquely through $R_\mathrm{red}$.
\end{cor}

\begin{proof}By Proposition \ref{225} and Corollary \ref{215}, $R_\mathrm{red}$ is cancellative and every surjective morphism from $R$ to a cancellative idempotent semiring factors uniquely through $R_\mathrm{red}$.  If $f:R\rightarrow A$ is any morphism to a cancellative idempotent semiring, then the surjective morphism $R\rightarrow f(R)$ factors uniquely through $R_\mathrm{red}$, and hence so does $f$.
\end{proof}

\begin{prop}Let $R$ be a simple idempotent semiring.  Let $f:R\rightarrow R_\mathrm{red}$ be the universal map.  Then $f(x)=f(y)$ if and only if every homomorphism $v:R\rightarrow \Gamma_{\mathrm{max}}$ into a totally ordered idempotent semifield satisfies $v(x)=v(y)$.
\end{prop}

\begin{proof}Given a prime congruence $\sim$, one may construct a map $R\rightarrow \mathop{Frac}(R/\mathnormal{\sim})$ into a totally ordered idempotent semifield.  Conversely given such a map $v:R\rightarrow \Gamma_{\mathrm{max}}$, one can define the prime congruence $\sim$ via $x\sim y$ when $v(x)=v(y)$.  This yields a canonically split surjection from the set of homomorphisms landing in totally ordered idempotent semifields to the set of prime congruences.
If $f(x)=f(y)$, then every prime congruence satisfies $x\sim y$, so every homomorphism $v:R\rightarrow \Gamma_{\mathrm{max}}$ into a totally ordered idempotent semifield satisfies $v(x)=v(y)$.  Conversely, suppose every homomorphism $v:R\rightarrow \Gamma_{\mathrm{max}}$ into a totally ordered idempotent semifield satisfies $v(x)=v(y)$.  Then every prime congruence satisfies $x\sim y$, so $f(x)=f(y)$.
\end{proof}




\begin{prop}\label{236}Let $R$ be a simple idempotent semiring.  Let $f:R\rightarrow R_\mathrm{red}$ be the universal map.  Then $f(x)=f(y)$ if and only if there exists some $s\neq 0$ in $R$ with $sx=sy$.
\end{prop}

\begin{proof}Suppose $f(x)=f(y)$.  Let $K=R_{(0)}$ be the localization of $R$ at the set of nonzero elements and $g:R\rightarrow K$ be the localization map.  Since $K$ is cancellative $g$ factors uniquely through $f$, so we write $g=hf$ for some $f$.  Then $\frac{x}{1}=g(x)=h(f(x))=g(y)=\frac{y}{1}$.  Hence there is some $s\neq 0$ with $sx=sy$.

Now suppose there is some nonzero $s\in R$ with $sx=sy$.  Then $f(s)f(x)=f(s)f(y)$.  Since $R_\mathrm{red}$ is cancellative and $\ker f=0$, we conclude $f(x)=f(y)$.
\end{proof}

Combining the above results yields the following.


\begin{thm}Let $R$ be a simple idempotent semiring and $x,y\in R$.  Then there exists some $s\neq 0$ in $R$ with $sx=sy$ if and only if every homomorphism $v:R\rightarrow \Gamma_{\max}$ whose target is a totally ordered idempotent semifield satisfies $v(x)=v(y)$.
\end{thm}



Noting that $sx\leq s1$ may be written as $s(x+1)=s(1)$, and that $v(x)\leq 1$ may be written as $v(x+1)=v(1)$, we obtain the following result, which essentially characterizes an intersection of valuation subsemirings.

\begin{cor}\label{334}Let $R$ be a simple idempotent semiring and $x\in R$.  Then there exists some $s\neq 0$ in $R$ with $sx\leq s$ if and only if every homomorphism $v:R\rightarrow \Gamma_{\max}$ whose target is a totally ordered idempotent semifield satisfies $v(x)\leq 1$.
\end{cor}

\section{Two notions of integrality}

\begin{defn}Let $R$ be an idempotent semiring.  Let $A$ be an $R$-algebra.  Let $x\in A$. Let $R\angles{x}$ be the smallest saturated subsemiring of $A$ containing both $x$ and the image of $R$ (or equivalently the saturated submodule generated by powers of $x$).  Then $x$ is said to be \emph{integral} over $R$ if $R\angles{x}$ is a {\sfinite} module over $R$.  $A$ is \emph{integral} over $R$ if every element of $A$ is integral over $R$.
\end{defn}

\begin{prop}Let $R$ be an idempotent semiring.  Let $A$ be an $R$-algebra.  Let $x\in A$.  $x$ is integral over $R$ if and only if there exists some $n>1$ and $c_0,\ldots,c_{n-1}\in R$ such that $x^n\leq c_0+c_1x+\ldots+c_{n-1}x^{n-1}$.
\end{prop}

\begin{proof}Suppose there exists some $n>1$ and $c_0,\ldots,c_{n-1}\in R$ such that $x^n\leq c_0+c_1x+\ldots+c_{n-1}x^{n-1}$.  Let $M$ be the saturated submodule of $A$ spanned by $1,x,\ldots,x^{n-1}$.  Then $c_0+c_1x+\ldots+c_{n-1}x^{n-1}\in M$ so $x^n\in M$.  Since $x,x^2,\ldots,x^n\in M$, multiplying any generator of $M$ by $x$ yields an element of $M$, i.e. $xM\subseteq M$.  Hence $x^k\in x^kM\subseteq M$ for $k\geq 0$.  From this, one easily checks that $R\angles{x}=M$.

Conversely, suppose $x$ is integral.  Because $R\angles{x}$ is {\sfinite}, there exists $u\in R\angles{x}$ such that for all $t\in R\angles{x}$ there exists $r\in R$ with $t\leq ru$.  Without loss of generality, we may replace $u$ with a larger element, and assume $u$ is a polynomial in $x$.  Write $u=a_0+\ldots+a_{n-1}x^{n-1}$ for some $n$, where the coefficients lie in $R$.  Then there exists $r$ with $x^n\leq ru=ra_0+\ldots+ra_{n-1}x^{n-1}$.
\end{proof}

\begin{defn}Let $R$ be an idempotent semiring.  An $R$-module $M$ is \emph{faithful} if $rM\neq 0$ for all $r\in R$.\end{defn}

\begin{defn}\label{362}Let $R$ be an idempotent semiring.  Let $A$ be an $R$-algebra.  Let $x\in A$.  Let $R\angles{x}$ be the smallest saturated subsemiring of $A$ containing both $x$ and the image of $R$.  $x$ is \emph{quasiintegral} over $R$ if there exists an $R\angles{x}$-submodule $M\subseteq A$ which is {\sfinite} as an $R$-module and which is faithful as an $R\angles{x}$-module.  $A$ is \emph{quasiintegral} over $R$ if every element of $A$ is quasiintegral over $R$.
\end{defn}

For the moment, we denote by $R[x]$ the smallest subsemiring of $A$ containing both $x$ and the image of $R$.  It is easy to see that the definition of quasiintegral does not change if instead of $R\angles{x}$ we use $R[x]$.

\begin{prop}\label{367}Let $R$ be an idempotent semiring.  Let $A$ be an $R$-algebra.  Let $x\in A$ be integral over $R$.  Then $x$ is quasiintegral over $R$.
\end{prop}

\begin{proof}Take $M=R\angles{x}$.
\end{proof}

In the theory of rings, there is a simple proof that quasiintegral elements are integral using the Cayley-Hamilton theorem.  Unfortunately, it is not clear whether there is a Cayley-Hamilton theorem for idempotent semirings.  We get around this by imposing some additional hypotheses, but hopefully these hypotheses are actually unneccessary.

\begin{lem}\label{375}Let $R$ be an idempotent semiring.  Let $A$ be an $R$-algebra.  Suppose furthermore that $A$ is simple.  Let $x\in A$ be quasiintegral over $R$.  One can choose $M$ as in Definition \ref{362} such that in addition, one has $R\angles{x}\subseteq M$.
\end{lem}

\begin{proof}Let $M$ be as in Definition \ref{362}.  Since $M$ is {\sfinite} over $R$, there is some $m\in M$ such that every element $\theta\in M$ satisfies $\theta\leq rm$ for some $r\in R$.  Since $M$ is a saturated submodule of $A$, any $\theta\in A$ with $\theta\leq rm$ satisfies $\theta\in M$.  Hence $M=\{\theta\in A\mid \exists r\in R\,\mathrm{such\,that}\, \theta\leq rm\}$.  Since $M$ is an $R\angles{x}$-submodule, for each $k$ we have $x^km\in M$ so $x^km\leq r_km$ for some $r_k\in R$.  Since $A$ is simple, we may choose $v\in A$ with $1\leq mv$.  Then for each $k\in\mathbb{N}$, $x^k\leq x^kvm\leq r_k(vm)$.  Let $M'$ be the saturated submodule of $A$ generated by $vm$.  Then $x^k\in M'$ for all $k$, and hence $R\angles{x}\subseteq M'$.  Note also that $M'$ is {\sfinite} as an $R$-module by construction.

Note also that since $x^km\leq r_k m$, $x^k(vm)\leq r_k (vm)$.  Pick $y\in M'$.  Then there is some $r\in R$ such that $y\leq rvm$.  Then for any $k$, $x^ky\leq rx^k(vm)\leq rr_k(vm)\in M'$.  If $z\in R\angles{x}$, we may write some inequality of the form $z\leq c_0+c_1x+\ldots+c_nx^n$ with $c_0,\ldots,c_n\in R$.  Since $c_ix^iy\in M'$ for each $i$, we have $(c_0+\ldots+c_nx^n)y\in M'$.  Since $zy\leq (c_0+\ldots+c_nx^n)y$, $zy\in M'$.  Hence $M'$ is an $R\angles{x}$-submodule.  It is faithful because it contains $1$.
\end{proof}



\begin{prop}\label{383}Let $R$ be an idempotent semiring.  Let $A$ be a simple $R$-algebra which is seminoetherian as an $R$-module.  Let $x\in A$ be quasiintegral over $R$.  Then $x$ is integral over $R$.
\end{prop}

\begin{proof}Use Lemma \ref{375} to pick an $R\angles{x}$-submodule $M\subseteq K$ with $R\angles{x}\subseteq M$ such that $M$ is {\sfinite} over $R$.  By the seminoetherian condition, $R\angles{x}$ is {\sfinite} as well.
\end{proof}

\section{Valuation semirings and quasiintegral closure}\label{s392}

\begin{defn}Let $R$ be an idempotent semiring.  Let $A$ be an $R$-algebra.  The \emph{quasiintegral closure} of $R$ in $A$ is the set of elements which are quasiintegral over $R$.
\end{defn}

\begin{defn}\label{411}Let $R$ be an idempotent semiring.  Let $M$  be an $R$-module.  The \emph{$R$-contraction} of $M$, denoted $M_{R\leq 1}$ is given by the following universal property.  There is a homomorphism $\phi:M\rightarrow M_{R\leq 1}$ such that $\phi(rx)\leq \phi(x)$ for all $r\in R$ and $x\in M$, and furthermore, if $\psi:A\rightarrow B$ satisfies $\psi (rx)\leq \psi(x)$ for all $r\in R$ and $x\in M$, then $\psi$ factors uniquely through $\phi$.
\end{defn}

\begin{lem}\label{317}Let $R$ be an idempotent semiring and $M$ be an $R$-module.  For any $x\in M$ (resp. $x\in R$) we use $\bar{x}$ to denote the image of $x$ in $M_{R\leq 1}$ (resp. $R_{R\leq 1}$).  Let $x,y\in M$.  Then $\bar{x}=\bar{y}$ if and only if there exists $r,s\in R$ with $x\leq ry$ and $sx\leq y$.
\end{lem}

\begin{proof}Write $x\sim y$ if there exists $r,s\in R$ with $x\leq ry$ and $sx\leq y$.  If $x\sim y$ and $y\sim z$, then we choose $r,r',s,s'$ with $x\leq ry$, $y\leq sx$, $y\leq r'z$, and $z\leq s'y$.  Then $x\leq rr'z$ and $z\leq ss'x$ so $x\sim z$.  It is easy to verify $\sim$ is symmetric and reflexive.  Suppose $x\sim y$ and $z\in M$.  Choose $r,s$ so $x\leq ry$ and $sx\leq y$.  Then $x+z\leq ry+z\leq (r+1)(y+z)$ and $y+z\leq (s+1)(x+z)$.  Hence $\sim$ is compatible with addition.  Now suppose $x\sim y$ and let $t\in R$.  Then we have $r,s\in R$ such that $tx\leq rty$ and $stx\leq ty$ so $tx\sim ty$.  Hence $\sim$ is compatible with scalar multiplication.  Hence the semimodule structure on $M/\mathnormal{\sim}$ is well-defined.  

Let $x\in M$ and $t\in R$.  Then $x\leq (t+1)x$ and $(t+1)x\leq (t+1)x$ so $x\sim (t+1)x$.  Let $y\in M/\mathnormal{\sim}$ be the image of $x$.   Then $y=ty+y$ so $ty\leq y$.  Hence the canonical map $\pi:M\rightarrow M/\mathnormal{\sim}$ factors uniquely through the universal map $\phi:M\rightarrow M_{R\leq 1}$.  If $\bar{x}=\bar{y}$ (i.e. $\phi(x)=\phi(y)$), this implies $\pi(x)=\pi(y)$ and so $x\sim y$.

Conversely, suppose $x\sim y$.  Let $r,s$ be such that $x\leq ry$ and $y\leq sx$.  Then $\bar{x}\leq \overline{ry}\leq \bar{y}$ and $\bar{y}\leq \overline{sx}\leq\bar{x}$.  Hence $\bar{x}=\bar{y}$.
\end{proof}


\begin{lem}\label{423}Let $R$ be an idempotent semiring and $M$ be an $R$-module.  For any $x\in M$ (resp. $x\in R$) we use $\bar{x}$ to denote the image of $x$ in $M_{R\leq 1}$ (resp. $R_{R\leq 1}$).  For $\bar{a}\in M_{R\leq 1}$, the construction $N=\{b\in M\mid \bar{b}\leq \bar{a}\}$ yields a {\sfinite} saturated $R$-submodule of $M$.  Furthermore every {\sfinite} saturated submodule has this form.
\end{lem}

\begin{proof}Let $\bar{a},\bar{b}\in M_{R\leq 1}$.  Observe that $\bar{b}\leq \bar{a}$ if and only if $\overline{a+b}=\bar{a}$.  This in turn occurs if and only if there exist $r,s\in R$ with $a+b\leq ra$ and $a\leq s(a+b)$.  Since we can always take $s=1$, this is equivalent to the statement that $a+b\leq ra$ for some $r\in R$.  If $a+b\leq ra$ then $b\leq ra$, while if $b\leq ra$ then $a+b\leq (r+1)a$.  Hence $a+b\leq ra$ for some $r\in R$ if and only if $b\leq ra$ for some $r\in R$.  Hence $\bar{b}\leq \bar{a}$ if and only if there is some $r\in R$ with $b\leq ra$.

Observe that $\{b\in M\mid \bar{b}\leq \bar{a}\}=\{b\in M\mid b\leq ra\,\mathrm{for\,some}\,r\in R\}$.  It is clear that this is a {\sfinite} saturated submodule and that every {\sfinite} saturated submodule has this form.
\end{proof}

\begin{prop}\label{399}Let $R$ be an idempotent semiring and $A$ be an $R$-algebra.  Suppose $A$ is simple.  For any $x\in A$ (resp. $x\in R$) we use $\bar{x}$ to denote the image of $x$ in $A_{R\leq 1}$ (resp. $R_{R\leq 1}$).  Then $x\in A$ is quasiintegral over $R$ if and only if there is some nonzero $s\in A_{R\leq 1}$ with $s\bar{x}\leq s$.
\end{prop}

\begin{proof}Suppose $s\bar{x}\leq s$.  Let $M=\{m\in A\mid \bar{m}\leq s\}$.  By lemma \ref{423}, $M$ is a saturated $R$-submodule of $K$ and is {\sfinite} as an $R$-module.  Note as well that if $m\in M$, then $\overline{xm}=\bar{x}\bar{m}\leq \bar{x}s\leq s$ so $xm\in M$.  As in the proof of Lemma \ref{375}, this implies that $M$ is an $R\angles{x}$-submodule of $A$.  $M$ is faithful because there are no zero divisors.  Hence $x$ is quasiintegral over $R$.

Suppose $x\in A$ is quasiintegral.  Choose a saturated $R\angles{x}$-submodule $M\subseteq A$ which is faithful as an $R\angles{x}$ submodule and {\sfinite} as an $R$-module.  There is some $s\in A_{R\leq 1}$ such that $M=\{m\in A\mid \bar{m}\leq s\}$.  Write $s=\bar{a}$.  Then $a\in M$ and so $xa\in M$.  Then $\bar{x}s=\bar{xa}\leq s$.  If $s=0$ then $M=0$ which contradicts the assumption of faithfulness.
\end{proof}

\begin{thm}\label{432}Let $R$ be an idempotent semiring.  Let $K$ be a simple $R$-algebra.  Let $x\in K$.  Then $x$ is quasiintegral over $R$ if and only if $v(\bar{x})\leq 1$ for every homomorphism $v:K_{R\leq 1}\rightarrow \Gamma_{\mathrm{max}}$ into a totally ordered idempotent semifield, where $\bar{x}\in K_{R\leq 1}$ is the class of $x$.
\end{thm}

\begin{proof}First note that since $K$ is simple, it has no zero divisors.  Combine Corollary \ref{334} and Proposition \ref{399} to conclude that $x$ is quasiintegral over $R$ if and only if $v(\bar{x})\leq 1$ for every homomorphism $v:K_{R\leq 1}\rightarrow \Gamma_{\mathrm{max}}$ into a totally ordered idempotent semifield.  
\end{proof}

Observing that homomorphisms $v:K_{R\leq 1}\rightarrow \Gamma_{\mathrm{max}}$ are in bijective correspondence with homomorphisms $v:K\rightarrow \Gamma_{\mathrm{max}}$ which have the property that $v(r)\leq 1$ for all $r\in R$, we obtain the following result.

\begin{cor}Let $R$ be an idempotent semiring.  Let $K$ be a simple $R$-algebra.  Let $x\in K$.  Then $x$ is quasiintegral over $R$ if and only if every homomorphism $v:K\rightarrow \Gamma_{\mathrm{max}}$ into a totally ordered idempotent semifield which satisfies $v(r)\leq 1$ for all $r\in R$ also satisfies $v(x)\leq 1$.
\end{cor}

\begin{thm}Let $R$ be an idempotent semiring.  Let $K$ be a unitgenerated $R$-algebra.  The quasiintegral closure of $R$ in $K$ is the intersection of all valuation subsemirings which contain $R$.
\end{thm}

We conclude this section with another application of Proposition \ref{399}.

\begin{prop}\label{307}Let $R$ be an idempotent semiring.  Let $A$ be an $R$-algebra which is a semifield.  Let $x\in A$ be quasiintegral over $R$.  Then $x$ is integral over $R$.
\end{prop}

\begin{proof}Since $A$ is a semifield, so is $A_{R\leq 1}$.  By Proposition \ref{399}, there is some nonzero $s\in A_{R\leq 1}$ such that $s\bar{x}\leq s$.  Since $s$ is a unit, $\bar{x}\leq 1$.  Equivalently, $\overline{x+1}=1$.  By Lemma \ref{317}, there is some $r\in R$ such that $x+1\leq r$.  This implies $x\leq r$.  This in turn implies that every element of $R\angles{x}$ is bounded by some element of $R$ so $R\angles{x}$ is generated by $1$.  Hence $x$ is integral.
\end{proof}

\section{The space of valuation orders}

\begin{defn}\label{448}Let $R$ be an idempotent semiring.  A \emph{valuation order} on $R$ is a relation $\preceq$ such that the following hold.

\begin{enumerate}
 \item\label{451} For all $x,y,z\in R$, if $x\preceq y\preceq z$ then $x\preceq z$.
 \item\label{452} For all $x,y\in R$, $x\preceq y$ or $y\preceq x$.
 \item\label{453} For all $x,y,z\in R$, if $x\preceq y$ then $x+z\preceq y+z$.
 \item\label{454} For all $x,y\in R$, if $x\leq y$ then $x\preceq y$.
 \item\label{455} For all $x,y,z\in R$, if $x\preceq y$ then $xz\preceq yz$.
 \item\label{456} For all $x,y,z\in R$, if $xz\preceq yz$ then either $x\preceq y$ or $z\preceq 0$.
\end{enumerate}

\end{defn}

\begin{defn}Let $X$ be a set.  We let $\mathcal{P}(X)$ be the set of subsets of $X$.  For any $x\in X$, we let $\mathrm{ev}_x:\mathcal{P}(X)\rightarrow \{0,1\}$ given by $\mathrm{ev}_x(Y)=1$ if $x\in Y$ and $0$ otherwise.  $\mathcal{P}(X)$ is equipped with the weakest topology such that each map $\mathrm{ev}_x$ is continuous.
\end{defn}

$\mathcal{P}(X)$ may be identified with a product of copies of $\{0,1\}$, so is compact.

\begin{defn}Let $R$ be an idempotent semiring.  We may view valuation orders on $R$ as subsets of $R\times R$ by identifying $\preceq$ with $\{(x,y)\in R^2\mid x\preceq y\}$.  The \emph{space of valuation orders} on $R$ is the subspace of $\mathcal{P}(R\times R)$ whose points are the valuation orders on $R$.
\end{defn}

A ring theoretic analogue of the following proposition is used as a lemma in the proof that adic spaces are spectral.  The space of valuation orders is roughly the same as the valuation spectrum $\mathop{Spv}R$ equipped with its constructible topology.

\begin{prop}Let $R$ be an idempotent semiring.  The space of valuation orders on $R$ is compact.
\end{prop}

\begin{proof}We write subsets of $R\times R$ as relations on $R$; In other words we identify a set $S\subseteq R\times R$ with the relation $\preceq$ defined by $x\preceq y$ if $(x,y)\in S$.  Hence $\mathcal{P}(R\times R)$ is the set of all relations on $R$.  Let $X$ be the space of valuation orders on $R$.  It suffices to show $X$ is closed in $\mathcal{P}(R\times R)$.  Note that $X=X_1\cap\ldots\cap X_6$ where $X_i\subseteq\mathcal{P}(R\times R)$ is the set of relations satisfying the $i$th axiom of Definition \ref{448}.  Hence it suffices to show each $X_i$ is closed.  For any $x,y\in R$, $\mathrm{ev}_{(x,y)}^{-1}(0)\cup\mathrm{ev}_{(y,z)}^{-1}(0)\cup\mathrm{ev}_{(x,z)}^{-1}(1)$ is a finite union of closed sets, so is closed. Then $X_1=\bigcap_{x,y,z\in R}(\mathrm{ev}_{(x,y)}^{-1}(0)\cup\mathrm{ev}_{(y,z)}^{-1}(0)\cup\mathrm{ev}_{(x,z)}^{-1}(1))$ is closed.  $X_2=\bigcap_{x,y\in R}(\mathrm{ev}_{(x,y)}^{-1}(1)\cup\mathrm{ev}_{(y,x)}^{-1}(1))$ is also closed.  $X_3=\bigcap_{x,y,z\in R} (\mathrm{ev}_{(x,y)}^{-1}(0)\cup\mathrm{ev}_{(x+z,y+z)}^{-1}(1))$ is likewise closed, as is $X_4=\bigcap_{x,y\in R;x\leq y}\mathrm{ev}_{(x,y)}^{-1}(1)$.  Similar arguments apply to $X_5$ and $X_6$.
\end{proof}

We now investigate the link between valuation orders and valuations.

\begin{prop}Let $\Gamma$ be a totally ordered abelian group.  Let $R$ be an idempotent semiring.  Let $v:R\rightarrow \Gamma_{\mathrm{max}}$ be a homomorphism.  Write $x\preceq y$ to mean $v(x)\leq v(y)$.  Then $\preceq$ is a valuation order.
\end{prop}

\begin{proof}Let $x,y,z\in R$.  If $v(x)\leq v(y)\leq v(z)$ then $v(x)\leq v(z)$ so Axiom \ref{451} of Definition \ref{448} holds.  Since $\Gamma_{\mathrm{max}}$ is totally ordered, $v(x)\leq v(y)$ or $v(y)\leq v(x)$ so Axiom \ref{452} holds.  If $v(x)=v(y)$, then $v(x+z)=v(x)+v(z)=v(y)+v(z)=v(y+z)$ so Axiom \ref{453} holds.  If $x\leq y$ then $v(x)\leq v(y)$ so Axiom \ref{454} holds.  If $v(x)=v(y)$ then $v(xz)=v(yz)$ so Axiom \ref{455} holds.  If $v(xz)\leq v(yz)$ then $v(x)v(z)\leq v(y)v(z)$, and since $\Gamma_{\mathrm{max}}$ is cancellative, either $v(x)\leq v(y)$ or $v(z)=0$.  Hence Axiom \ref{456} holds.
\end{proof}

\begin{prop}Let $R$ be an idempotent semiring.  Let $\preceq$ be a valuation order.  Let $\sim$ be the relation given by declaring that $x\sim y$ when $x\preceq y\preceq x$.  Then $R/\mathnormal{\sim}$ is a totally ordered cancellative idempotent semiring.  Furthermore, the canonical order on $R/\mathnormal{\sim}$ agrees with that induced by $\preceq$.
\end{prop}

\begin{proof}Because $\preceq$ is a preorder, $\sim$ is an equivalence relation.  Axioms \ref{453} and \ref{455} of Definition \ref{448} then imply that $\sim$ is a congruence.  Hence $R/\mathnormal{\sim}$ is an idempotent semiring.  For any element $x\in R$ we shall write $\bar{x}$ for the corresponding element of $R/\mathnormal{\sim}$.  

Let $x,y\in R$.  Suppose $\bar{x}\leq \bar{y}$. Then $\overline{x+y}=\bar{y}$.  Hence there is some $z\sim y$ such that $x\leq z$ (e.g. take $z=x+y$).  Then $x\preceq z\preceq y$ by Axiom \ref{454}.  Conversely, suppose $x\preceq y$.  Then $x+y\preceq y+y=y$.  On the other hand $y\leq x+y$ so $y\preceq x+y$.  Hence $x+y\sim y$. This means $\bar{x}+\bar{y}=\bar{y}$, and hence $\bar{x}\leq \bar{y}$.  This implies that the order induced by $\preceq$ is the canonical order.

Axiom \ref{456} implies that if $xz\sim yz$ then $x\sim y$ or $z\preceq 0$.  If $z\preceq 0$ then $\bar{z}\leq 0$ so $\bar{z}= 0$.  Hence if $\bar{x}\bar{z}=\bar{y}\bar{z}$ then $\bar{x}=\bar{y}$ or $\bar{z}=0$.  Hence $R/\mathnormal{\sim}$ is cancellative.  Axiom \ref{452} implies that $R/\mathnormal{\sim}$ is totally ordered.
\end{proof}

\begin{cor}Let $R$ be an idempotent semiring.  Let $\preceq$ be a valuation order.  Then there is an totally ordered abelian group $\Gamma$ and a homomorphism $v:R\rightarrow \Gamma_{\mathrm{max}}$ such that $x\preceq y$ if and only if $v(x)\leq v(y)$.
\end{cor}

\begin{proof}Let $\sim$ be the equivalence relation induced by $\preceq$.  Then $R/\mathnormal{\sim}$ is a totally ordered cancellative idempotent semiring and $\mathop{Frac}(R/\mathnormal{\sim})$ is a totally ordered idempotent semifield.  Let $\Gamma=\mathop{Frac}(R/\mathnormal{\sim})^\times$ so $\mathop{Frac}(R/\mathnormal{\sim})=\Gamma_{\mathrm{max}}$.  Define $v$ to be the composition of the quotient map $\pi:R\rightarrow R/\mathnormal{\sim}$ and the inclusion $i:R/\mathnormal{\sim}\rightarrow \mathop{Frac}(R/\mathnormal{\sim})$.  Then $v(x)\leq v(y)$ if and only if $\pi(x)\leq \pi(y)$ which in turn holds if and only if $x\preceq y$.
\end{proof}

\begin{defn}Let $R$ be an idempotent semiring.  Let $S\subseteq R\times R$.  $S$ is \emph{admissible} if there is a totally ordered abelian group $\Gamma$, and a homomorphism $v:R\rightarrow \Gamma_{\mathrm{max}}$ such that $v(y)<v(x)$ for all $(x,y)\in S$.
\end{defn}

\begin{thm}\label{505}Let $R$ be an idempotent semiring.  Let $S\subseteq R\times R$.  Suppose each finite subset of $S$ is admissible.  Then $S$ is admissible.
\end{thm}

\begin{proof}Let $T\subseteq S$ be finite.  Let $w_T$ be a homomorphism into a totally ordered idempotent semifield such that $w_T(y)<w_T(x)$ for all $(x,y)\in T$.  Then for all $(x,y)\in T$, we have $w_T(x)\not\leq w_T(y)$.  Let $\preceq_T$ be given by $x\preceq_T y$ if $w_T(x)\leq w_T(y)$.  Then for all $(x,y)\in T$, we have $x\not\preceq_T y$ so $\mathrm{ev}_{(x,y)}(\preceq_T)=0$.  Hence $\preceq_T\in \bigcap_{(x,y)\in T} \mathrm{ev}_{(x,y)}^{-1}(0)$.

For any finite subset $T\subseteq S$, we define a subset of the space of valuation orders by $A_T=\bigcap_{(x,y)\in T} \mathrm{ev}_{(x,y)}^{-1}(0)$.  The above paragraph shows each $A_T$ is nonempty.  Furthermore $A_{T_1}\cap\ldots\cap A_{T_n}=A_{T_1\cup\ldots\cup T_n}\neq \emptyset$.  Hence any finite intersection of the sets $A_T$ is nonempty.  By construction, $A_T$ is closed.  By the compactness of the space of valuation orders, $\bigcap_{T\subseteq S\,\mathrm{finite}} A_T\neq\emptyset$.  It is not hard to see that $\bigcap_{T\subseteq S\,\mathrm{finite}}A_T=\bigcap_{(x,y)\in S}\mathrm{ev}_{(x,y)}^{-1}(0)$, so this subset is nonempty.

Let $\preceq$ be a valuation order which lies in $\mathrm{ev}_{(x,y)}^{-1}(0)$ for all $(x,y)\in S$.  Let $w$ be a homomorphism into a totally ordered idempotent semifield such that $x\preceq y$ if and only if $w(x)\leq w(y)$.  Then for all $(x,y)\in S$ we have $x\not\preceq y$ so $w(x)\not\leq w(y)$ and $w(y)<w(x)$.  Hence $S$ is admissible.
\end{proof}

\section{Extensions of valuations}

We first consider extension of valuations in the case of an extension of idempotent semifields $K\rightarrow L$ where $K$ is totally ordered and the valuation on $K$ is the identity map.

\begin{lem}\label{519}Let $K$ be a totally ordered idempotent semifield.  Let $L$ be an idempotent semifield equipped with a homomorphism $f:K\rightarrow L$.  Let $\mathcal{O}_K=\{x\in K\mid x\leq 1\}$ and let $\mathcal{O}_L$ be its quasiintegral closure in $L$.  Then $\mathcal{O}_L=\{x\in L\mid x\leq 1\}$.  Furthermore $\mathcal{O}_L^\times=\{1\}$.
\end{lem}

\begin{proof}Let $y\in\mathcal{O}_L$.  Then $y$ is quasiintegral over $\mathcal{O}_K$.  Since each element of $\mathcal{O}_K$ is less than or equal to one, we apply Proposition \ref{399} to see that there is some nonzero $s\in L$ with $sy\leq s$.  Multiplying both sides by $s^{-1}$ yields $y\leq 1$.  If furthermore, $y\in\mathcal{O}_L^\times$, we also get $y^{-1}\leq 1$ so $y=1$.
\end{proof}

\begin{thm}\label{440}Let $K$ be a totally ordered idempotent semifield.  Let $L$ be an idempotent semifield equipped with an injective homomorphism $f:K\rightarrow L$.  Then there is a totally ordered idempotent semifield $E$ and a morphism $v:L\rightarrow E$ such that the composite map $v\circ f:K\rightarrow E$ is injective.
\end{thm}

\begin{proof}Let $\mathcal{O}_K=\{x\in K\mid x\leq 1\}$ and let $\mathcal{O}_L$ be its quasiintegral closure in $L$.  Note that $L_{\mathcal{O}_K\leq 1}=L$.  Apply Theorem \ref{432} to see that $\mathcal{O}_L=\bigcap_{w} \{x\in L\mid w(x)\leq 1\}$ where $w$ ranges over all homomorphisms from $L$ into a totally ordered idempotent semifield.  Apply Lemma \ref{519} to see that $\mathcal{O}_L=\{x\in L\mid x\leq 1\}$ and in particular that $\mathcal{O}_L^\times=1$.

Let $U\subseteq \mathcal{O}_K$ be finite and suppose $1\not\in U$.  I claim there is a homomorphism $w_U$ from $L$ into a totally ordered idempotent semifield such that $w_U(f(u))<1$ for all $u\in U$.  Observe that $u<1$ for all $u\in U$.  Let $t=\sum_{u\in U} u$, and note $t<1$.  Suppose for the moment that every homomorphism $w$ from $L$ into a totally ordered idempotent semifield satisfies $w(f(t))=1$.  Then $w(f(t^{-1}))\leq 1$ for all such $w$, so $f(t^{-1})\in \mathcal{O}_L$ and $f(t)\in \mathcal{O}_L$.  Then $f(t)\in\mathcal{O}_L^\times$ so $f(t)=1$.  Since $f$ is injective, $t=1$, contradicting $t<1$.  Hence we may find some homomorphism $w_U:L\rightarrow (\Gamma_U)_\mathrm{max}$ into a totally ordered idempotent semifield such that $w_U(f(t))\neq 1$.  For each $u\in U$, $u\leq 1$ so $w_U(f(u))\leq 1$.  If $w_U(f(u))=1$ for some $u$ then $w_U(f(t))=\sum_{u\in U}w_U(f(u))=1$, a contradiction.  Hence $w_U(f(u))<1$.

Let $S\subseteq L\times L$ be given by $S=\{(1,f(x))\mid x<1\}$.  Let $T\subseteq S$ be finite.  Then $T=\{(1,f(u))\mid u\in U\}$ for some finite set $U\in K$.  Furthermore, each $u\in U$ satisfies $u<1$ so $U\subseteq \mathcal{O}_K$ and $1\not\in U$.  Choose a homomorphism $w_U$ such that $w_U(f(u))<1=w_U(1)$ for all $u\in U$.  Then $w_U(y)<w_U(x)$ for all $(x,y)\in T$.  Hence $T$ is admissible.  By Theorem \ref{505}, $S$ is admissible.  Hence there is some homomorphism $v:L\rightarrow E$ into a totally ordered idempotent semifield $E$ such that $v(f(x))<1$ for all $x<1$.

It remains to show $v\circ f$ is injective.  Let $x,y\in K$ and suppose $v(f(x))=v(f(y))$.  Since $K$ is totally ordered, either $x\leq y$ or $y\leq x$.  Without loss of generality, we assume $x\leq y$.  If $y=0$ then $x=0$ and we are done.  Otherwise, $xy^{-1}\leq 1$.  If $xy^{-1}=1$, we are done, so we may suppose $xy^{-1}<1$.  Then $v(f(xy^{-1}))<1$ and hence $v(f(x))<v(f(y))$, which is a contradiction.  Hence $v\circ f$ is injective.
\end{proof}

\begin{defn}Let $R$ be an idempotent semiring and $A$ be an $R$-algebra.  $A$ is an \emph{extensible} $R$-algebra if every element of $A$ which is quasiintegral over $R$ is integral over $R$.
\end{defn}

Propositions \ref{383} and \ref{307} give some sufficient conditions for an $R$-algebra to be extensible.

\begin{lem}Let $R$ be an idempotent semiring and $A$ be an $R$-algebra.  Let $x\in R$ and suppose $x$ is a unit in $A$.  Suppose furthermore that $A$ is integral over $R$.  Then there exists $y\in R$ with $xy\geq 1$.
\end{lem}

\begin{proof}Let $u\in A$ be the inverse of $x$.  Then $u$ is integral over $R$ so $u^n\leq c_0+\ldots +c_{n-1}u^{n-1}$ for some $n$ and some $c_0,\ldots,c_{n-1}\in R$.  This yields $1\leq c_0x^n+\ldots+c_{n-1}x$.  Let $y=c_0x^{n-1}+\ldots+c_{n-1}\in R$.  Then $1\leq xy$.
\end{proof}

\begin{prop}Let $L$ be a simple idempotent semiring and let $K\subseteq L$ be a subsemiring.  Suppose $K$ is a totally ordered idempotent semifield.  Suppose furthermore that $L$ is extensible as an algebra over $\mathcal{O}_K=\{t\in K\mid t\leq 1\}$.  Then the induced map $f:K\rightarrow L_{\mathrm{red}}$ is injective.
\end{prop}

\begin{proof}Suppose $x,y\in K$ satisfy $f(x)=f(y)$.  If $x=0$ then $y=0$ and vice versa, since $\ker f$ is a saturated ideal in a semifield.  Otherwise $x$ and $y$ are units.  Since $K$ is totally ordered, either $x\leq y$ or $y\leq x$.  We suppose without loss of generality that $x\leq y$ so $xy^{-1}\in \mathcal{O}_K$.  Since $f(x)=f(y)$, there is some $s\in L$ with $sx=sy$, which implies $s(xy^{-1})\leq s$ and $s(yx^{-1})\leq s$.   Let $\mathcal{O}_L$ be the quasiintegral closure of $\mathcal{O}_K$.  $\mathcal{O}_L$ is the set of $t\in L$ satisfying $st\leq s$ for some $s\in L$, which includes both $xy^{-1}$ and $yx^{-1}$.  Hence $xy^{-1}\in\mathcal{O}_L^\times\cap \mathcal{O}_K$.  Because $L$ is extensible as an $\mathcal{O}_K$-algebra, $\mathcal{O}_L$ is integral over $\mathcal{O}_K$.  Hence there is some $t\in\mathcal{O}_K$ such that $(xy^{-1})t\geq 1$.  Since $1$ is maximal in $\mathcal{O}_K$, $t\leq 1$ so $xy^{-1}\geq xy^{-1}t\geq 1$.  This implies $xy^{-1}=1$ so $x=y$.  Hence $f$ is injective.
\end{proof}

\begin{cor}\label{471}Let $L$ be a simple idempotent semiring and let $K\subseteq L$ be a subsemiring.  Suppose $K$ is a totally ordered idempotent semifield.  Suppose furthermore that $L$ is extensible as an algebra over $\mathcal{O}_K=\{t\in K\mid t\leq 1\}$.  Then the induced map $g:K\rightarrow L_{(0)}$ is injective.
\end{cor}

\begin{proof}Let $f:K\rightarrow L_{\mathrm{red}}$ be the map induced by the inclusion $K\subseteq L$.  Suppose $g(x)=g(y)$.  Then $sx=sy$ for some nonzero $s$.  Then Proposition \ref{236} implies $f(x)=f(y)$.  Since $f$ is injective, $x=y$ so $g$ is injective.
\end{proof}

\begin{prop}\label{476}Let $L$ be a simple idempotent semiring and let $K\subseteq L$ be a subsemiring.  Suppose $K$ is a totally ordered idempotent semifield.  Suppose furthermore that $L$ is extensible as an algebra over $\mathcal{O}_K=\{t\in K\mid t\leq 1\}$.  Let $i:K\rightarrow L$ be the inclusion.  Then there is a homomorphism $v:L\rightarrow E$ into a totally ordered idempotent semifield $E$ such that $v\circ i$ is injective.
\end{prop}

\begin{proof}Let $u:L\rightarrow L_{(0)}$ be the canonical map.  Corollary \ref{471} implies $u\circ i$ is injective.  Applying Theorem \ref{440} gives a homomorphism $w:L_{(0)}\rightarrow E$ into a totally ordered idempotent semifield such that $w\circ(u\circ i)$ is injective.  Take $v=w\circ u$.
\end{proof}

We would like to relax the hypothesis that $K$ is a totally ordered idempotent semifield.  To do this we study the $\mathcal{O}_K$-contraction.

\begin{lem}\label{484b}Let $K$ be a unitgenerated idempotent semiring.  Fix a surjective homomorphism $v:K\rightarrow\Gamma_{\mathrm{max}}$ into a totally ordered idempotent semifield.  Let $\mathcal{O}_K=\{x\in K\mid v(x)\leq 1\}$.  Then $v$ induces an isomorphism $K_{\mathcal{O}_K\leq 1}\cong \Gamma_{\mathrm{max}}$.
\end{lem}

\begin{proof}Since each element of $\mathcal{O}_K$ satisfies $v(x)\leq 1$, $v$ induces a homomorphism $f:K_{\mathcal{O}_K\leq 1}\rightarrow\Gamma_{\mathrm{max}}$.  $f$ is surjective since $v$ is surjective.  Suppose $f(\bar{x})=f(\bar{y})$ and $x,y\neq 0$.  Then $v(x)=v(y)$.  Write $x$ as a finite sum of units $x=\sum_{i\in I} x_i$.  There is some $k\in I$ such that $v(x)=v(x_k)$.  Then $v(x_k^{-1}y)=1$ so $x_k^{-1}y\in \mathcal{O}_K$.  Let $r=x_k^{-1}y$.  Then $y=rx_k\leq rx$.  A similar argument yields an element $s\in \mathcal{O}_K$ such that $x\leq sy$.  Hence $\bar{x}=\bar{y}$ by Lemma \ref{317}.
\end{proof}

\begin{lem}\label{490b}Let $R$ be an idempotent semiring.  Let $A$ and $B$ be $R$-algebras.  Suppose $A$ is a $R$-subalgebra of $B$ and let $i$ be the inclusion.  Then the map $A_{R\leq 1}\rightarrow B_{R\leq 1}$ induced by $i$ is injective.
\end{lem}

\begin{proof}Let $x,y\in A$.  Suppose $\overline{i(x)}=\overline{i(y)}$ as elements of $B_{R\leq 1}$.  Then there exists $r,s\in R$ with $i(x)\leq ri(y)$ and $i(y)\leq si(x)$.  Since $i$ is injective, $x\leq ry$ and $y\leq sx$.  Hence $\bar{x}=\bar{y}$ as elements of $A_{R\leq 1}$.
\end{proof}

\begin{lem}\label{497}Let $R$ be an idempotent semiring.  Let $A$ be an extensible $R$-algebra which is simple.  Then $A_{R\leq 1}$ is an extensible $R_{R\leq 1}$-algebra.
\end{lem}

\begin{proof}Let $\bar{x}\in A_{R\leq 1}$ be the class of an element $x\in A$.  Suppose $\bar{x}$ is quasiintegral over $R_{R\leq 1}$.  Then there is a {\sfinite} saturated submodule $M\subseteq A_{R\leq 1}$ such that $xM\subseteq M$.  Choose $\bar{s}\in M$ so $M=\{\bar{a}\in A_{R\leq 1}\mid \bar{a}\leq \bar{s}\}$.  Then $\bar{s}\bar{x}\leq \bar{s}$.  Apply Proposition \ref{399} to conclude $x$ is quasiintegral over $R$.  Since $A$ is an extensible $R$-algebra, $x$ is integral over $R$.  We have some $n>1$ and $c_0,\ldots,c_{n-1}\in R$ such that $x^n\leq c_0+\ldots+c_{n-1}x^{n-1}$.  Then $\bar{x}^n\leq\bar{c}_0+\ldots+\bar{c}_{n-1}\bar{x}^{n-1}$, so that $\bar{x}$ is integral over $R_{R\leq 1}$.
\end{proof}

\begin{lem}\label{503}Let $K$ be a unitgenerated idempotent semiring equipped with a surjective homomorphism $v:K\rightarrow\Gamma_\mathrm{max}$ into a totally ordered idempotent semifield.  Let $\mathcal{O}_K=\{x\in K\mid v(x)\leq 1\}$.  Then the inclusion of $\mathcal{O}_K$ into $K$ yields an isomorphism $(\mathcal{O}_K)_{\mathcal{O}_K\leq 1}\cong \{\bar{x}\in K_{\mathcal{O}_K\leq 1}\mid \bar{x}\leq 1\}$.
\end{lem}

\begin{proof}Since the inclusion of $\mathcal{O}_K$ into $K$ is injective, the same is true for the induced map $j:(\mathcal{O}_K)_{\mathcal{O}_K\leq 1}\rightarrow K_{\mathcal{O}_K\leq 1}$.  It remains to determine the image of this map.  Let $\bar{x}\in (\mathcal{O}_K)_{\mathcal{O}_K\leq 1}$.  Then $\bar{x}\leq 1$ so $j(\bar{x})\leq 1$.  Conversely, let $\bar{y}\in K_{\mathcal{O}_K\leq 1}$ and suppose $\bar{y}\leq 1$.  Then there exists $r\in\mathcal{O}_K$ such that $y\leq r$.  Hence $y\in \mathcal{O}_K$.  Hence $\bar{y}$ is in the image of $j$.
\end{proof}

\begin{thm}\label{518}Let $K$ be a unitgenerated idempotent semiring.  Let $v:K\rightarrow E$ be a surjective homomorphism into some totally ordered idempotent semifield.  Let $L$ be a unitgenerated idempotent semiring containing $K$, and let $i:K\rightarrow L$ be the inclusion.  Let $\mathcal{O}_K=\{x\in K\mid v(x)\leq 1\}$.  Suppose $L$ is extensible as an $\mathcal{O}_K$-algebra.  Then there is a totally ordered idempotent semifield $F$, an injective homomorphism $j:E\rightarrow F$, and a homomorphism $w:L\rightarrow F$ such that $wi=jv$.
\end{thm}

\begin{proof}Let $K'=K_{\mathcal{O}_K\leq 1}$ and $L'=L_{\mathcal{O}_K\leq 1}$.  Let $i':K'\rightarrow L'$ be induced by $i$.  Then by Lemma \ref{490b}, $i'$ is injective.  Let $v':K'\rightarrow E$ be induced by $v$.  By Lemma \ref{484b} $v'$ is an isomorphism, so in particular $K'$ is a totally ordered idempotent semifield.  By Lemma \ref{503}, we may identify $\mathcal{O}_{K'}=\{x\in K'\mid v'(x)\leq 1\}=\{x\in K'\mid x\leq 1\}$ with $(\mathcal{O}_K)_{\mathcal{O}_K\leq 1}$.  By Lemma \ref{497}, $L'$ is extensible as an $\mathcal{O}_{K'}$-algebra.  Proposition \ref{476} gives a homomorphism $w':L'\rightarrow F$ into a totally ordered idempotent semifield such that $w'i'$ is injective.  Let $\pi_K:K\rightarrow K'$ and $\pi_L:L\rightarrow L'$ be the canonical maps.  Observe $v'\pi_K=v$ and $\pi_Li=i'\pi_K$.  Define $j:E\rightarrow F$ by $j=w'i'(v')^{-1}$.  Define $w:L\rightarrow F$ by $w=w'\pi_L$.  Clearly $j$ is injective since this is true for both $(v')^{-1}$ and $w'i'$.  It remains to check that $wi=jv$.  But $jv=jv'\pi_K=w'i'v'^{-1}v'\pi_K=w'i'\pi_K=w'\pi_Li=wi$.
\end{proof}

\begin{cor}\label{608}Let $K$ be a unitgenerated idempotent semiring.  Let $v:K\rightarrow E$ be a surjective homomorphism into some totally ordered idempotent semifield.  Let $L$ be an idempotent semifield containing $K$, and let $i:K\rightarrow L$ be the inclusion.  Then there is a totally ordered idempotent semifield $F$, an injective homomorphism $j:E\rightarrow F$, and a homomorphism $w:L\rightarrow F$ such that $wi=jv$.
\end{cor}

\begin{cor}Let $K$ be a unitgenerated idempotent semiring.  Let $v:K\rightarrow E$ be a surjective homomorphism into some totally ordered idempotent semifield.  Let $L$ be a unitgenerated idempotent semiring containing $K$, and let $i:K\rightarrow L$ be the inclusion.  Let $\mathcal{O}_K=\{x\in K\mid v(x)\leq 1\}$.  Suppose $L$ is seminoetherian as an $\mathcal{O}_K$-algebra.  Then there is a totally ordered idempotent semifield $F$, an injective homomorphism $j:E\rightarrow F$, and a homomorphism $w:L\rightarrow F$ such that $wi=jv$.
\end{cor}

\section{Extension of valuations for rings}

\begin{defn}\label{627}Let $R$ be a ring and $M$ be an $R$-module.  Let $\mathcal{S}_f(M,R)$ be the idempotent semigroup of submodules of $M$.
\end{defn}

If $M$ is an $R$-algebra then $\mathcal{S}_f(M,R)$ is actually a semiring with $R\subseteq M$ as multiplicative identity.  Here the product of finitely generated submodules $N,P\subseteq M$ is the submodule generated by $\{np\mid n\in N,p\in P\}$.

The following two results may be found in \cite{macpherson}.

\begin{prop}\label{535}Let $R$ be a ring and $A$ an $R$-algebra.  Let $\Gamma$ be a totally ordered abelian group.  Then there is a bijective correspondence between homomorphisms $\mathcal{S}_f(A,R)\rightarrow \Gamma_{\max}$ and valuations $A\rightarrow \Gamma_{\max}$ which are bounded by $1$ on $R$.  Under this correspondence, a homomorphism $f:\mathcal{S}_f(A,R)\rightarrow \Gamma_{\max}$ corresponds to the valuation $x\mapsto f(xR)$.
\end{prop}

\begin{prop}\label{534}Let $R$ be a ring and $M$ an $R$-module.  There is a bijective correspondence between saturated subsemigroups of $\mathcal{S}_f(M,R)$ and submodules of $M$.  A submodule $N\subseteq M$ corresponds to $\{N'\subseteq N\mid N'\,\mathrm{is\, finitely\, generated}\}$.  
\end{prop}

With these results in mind, the classical extension of valuations theorem states that given a valued field $K$ and an extension $L/K$, one may extend the corresponding homomorphism $\mathcal{S}_f(K,\mathcal{O}_K)\rightarrow\Gamma_{\max}$ to $\mathcal{S}_f(L,\mathcal{O}_K)$.  We will give another proof of this theorem below.  The difficult part is to show $\mathcal{S}_f(L,\mathcal{O}_K)$ is an extensible $\mathcal{S}_f(\mathcal{O}_K,\mathcal{O}_K)$ algebra.  The following lemma is a consequence of the Cayley-Hamilton theorem; this is the only place we will use any deep results of commutative algebra.

\begin{lem}Let $R$ be a ring and $A$ an $R$-algebra.  Let $x\in A$.  Suppose there is a finitely generated $R$-submodule $M\subseteq A$ such that $xM\subseteq M$ and $1\in M$.  Then $x$ is a zero of a monic polynomial $f\in R[T]$.
\end{lem}

\begin{proof}By the Cayley-Hamilton theorem, there is a monic polynomial $f(T)$ such that $f(x)M=0$.  Since $1\in M$, $f(x)=0$.
\end{proof}

Another fact we will use is that fields correspond to unitgenerated semirings.

\begin{lem}\label{548}Let $R$ be a ring and $A$ be an $R$-algebra.  If $A$ is a field then $\mathcal{S}_f(A,R)$ is unitgenerated.
\end{lem}

\begin{proof}$\mathcal{S}_f(A,R)$ is generated by principal submodules, and the submodule generated by a single unit is invertible.
\end{proof}

\begin{prop}\label{554}Let $R$ be a ring and let $A$ be an $R$-algebra which is a field.  Then $\mathcal{S}_f(A,R)$ is an extensible $\mathcal{S}_f(R,R)$-module.
\end{prop}

\begin{proof}Let $x\in \mathcal{S}_f(A,R)$ be quasiintegral over $\mathcal{S}_f(R,R)$.  Suppose $x$ is principal, i.e. $x=\hat{x}R$ for some $\hat{x}\in A$.  By Lemma \ref{375}, there is some saturated $\mathcal{S}_f(R,R)$-submodule (and hence subsemigroup) $M$ with $\mathcal{S}_f(R,R)\subseteq M \subseteq \mathcal{S}_f(A,R)$ such that $xM\subseteq M$.  This yields an $R$-module $\hat{M}$ with $R\subseteq \hat{M}\subseteq A$ and $\hat{x}\hat{M}\subseteq \hat{M}$.  Hence $\hat{x}$ is a zero of a monic polynomial $f\in R[T]$.  Then $\hat{x}^n$ is in the $R$-submodule generated by $1,\ldots,\hat{x}^{n-1}$ for some $n$.  This implies that $x^n\leq 1+\ldots+x^{n-1}$ so $x$ is integral over $\mathcal{S}_f(R,R)$.
\end{proof}

\begin{thm}Let $K$ be a field.  Fix a valuation $v:K\rightarrow \Gamma\cup\{0\}$ where $\Gamma$ is a totally ordered abelian group.  Let $L$ be an extension of $K$.  Then there is a totally ordered abelian group $\Gamma'\supseteq \Gamma$ and a valuation $w:L\rightarrow \Gamma'\cup\{0\}$ such that $v=w\mid_K$.
\end{thm}

\begin{proof}Let $\mathcal{O}_K=\{x\in K\mid v(x)\leq 1\}$.  For any $\mathcal{O}_K$-algebra $A$, we let $\tilde{A}$ denote $\mathcal{S}_f(A,\mathcal{O}_K)$.  By Lemma \ref{548}, $\tilde{K}$ and $\tilde{L}$ are unitgenerated.  Observe that $\mathcal{O}_K$ is the multiplicative identity in $\tilde{K}$.  By Proposition \ref{534}, $\mathcal{O}_{\tilde{K}}=\{x\in\tilde{K}\mid x\leq 1\}$ is the subsemigroup of $\tilde{K}$ corresponding to $\mathcal{O}_K$, i.e. $\mathcal{O}_{\tilde{K}}=\tilde{\mathcal{O}}_K$.  By Proposition \ref{554}, we know that $\tilde{L}$ is an extensible $\tilde{\mathcal{O}}_K$-algebra, so is an extensible $\mathcal{O}_{\tilde{K}}$-module.  By Proposition \ref{535} $v$ yields a homomorphism $\tilde{v}:\tilde{K}\rightarrow \Gamma_{\max}$.  By Theorem \ref{518}, we obtain a homomorphism $\tilde{w}:\tilde{L}\rightarrow \Gamma'_{\max}$ extending $\tilde{v}$ for some totally ordered abelian group $\Gamma'_{\max}$.  Applying Proposition \ref{535} again yields the result. 
\end{proof}

\bibliography{subobjbib}

\begin{thebibliography}{1}

\bibitem{arithmeticsite}
Alain Connes and Caterina Consani.
\newblock The arithmetic site.
\newblock {\em Comptes Rendus Mathematique}, 352(12):971--975, 2014.

\bibitem{scalingsite}
Alain Connes and Caterina Consani.
\newblock The scaling site.
\newblock {\em Comptes Rendus Mathematique}, 354(1):1--6, 2016.

\bibitem{giansiracusa}
Jeffrey Giansiracusa and Noah Giansiracusa.
\newblock Equations of tropical varieties.
\newblock {\em arXiv preprint arXiv:1308.0042}, 2013.

\bibitem{joomin}
D{\'a}niel Jo{\'o} and Kalina Mincheva.
\newblock Prime congruences of idempotent semirings and a nullstellensatz for
  tropical polynomials.
\newblock {\em arXiv preprint arXiv:1408.3817}, 2014.

\bibitem{macpherson}
Andrew~W Macpherson.
\newblock Skeleta in non-archimedean and tropical geometry.
\newblock {\em arXiv preprint arXiv:1311.0502}, 2013.

\end{thebibliography}
\bibliographystyle{plain}

\end{document}